\documentclass[12pt]{amsart}

\usepackage[usenames]{color}
\usepackage{fullpage,url,mathrsfs,stmaryrd}

\usepackage{relsize}
\usepackage[bbgreekl]{mathbbol}
\usepackage{amsfonts}

\DeclareSymbolFontAlphabet{\mathbb}{AMSb}
\DeclareSymbolFontAlphabet{\mathbbl}{bbold}

\newcommand{\BB}{{\mathbb{B}}}
\newcommand{\BG}{{\mathbb{G}}}

\usepackage{amssymb}
\usepackage[all]{xy}
\usepackage{mathrsfs}
\usepackage{enumerate}
\usepackage{amscd}

\usepackage{bbm}

\DeclareMathOperator{\Syn}{{Syn}}

\newcommand{\AAut}{\underline{\on{Aut}}}

\newcommand{\MMor}{\underline{\on{Mor}}}

\DeclareMathOperator{\Ob}{{Ob}}
\DeclareMathOperator{\Ab}{{Ab}}

\newcommand{\SR}{{{}^s \sR}}

\newcommand{\cA}{{\mathcal A}}

\newcommand{\cC}{{\mathcal C}}

\newcommand{\sG}{{\mathscr G}}

\newcommand{\sR}{{\mathscr R}}

\newcommand{\sX}{{\mathscr X}}

\newcommand{\fB}{{\mathfrak B}}

\newcommand{\fG}{{\mathfrak G}}

\newcommand{\fL}{{\mathfrak L}}

\newcommand{\nc}{\newcommand}

\nc\wh{\widehat}

\nc\on{\operatorname}

\nc\Gr{\on{Gr}}

\nc\Fl{\on{Fl}}

\newtheorem{cor}[subsubsection]{Corollary}
\newtheorem{lem}[subsubsection]{Lemma}
\newtheorem{prop}[subsubsection]{Proposition}

\theoremstyle{remark}

\newcommand{\BF}{{\mathbb{F}}}
\newcommand{\BN}{{\mathbb{N}}}

\newcommand{\BZ}{{\mathbb{Z}}}

\DeclareMathOperator{\Isom}{{Isom}}
\DeclareMathOperator{\Lie}{{Lie}}

 \DeclareMathOperator{\Spf}{{Spf}}

\DeclareMathOperator{\Cone}{{Cone}}

\newcommand{\limto}{{\displaystyle\lim_{\longrightarrow}}}
\newcommand{\rightlim}{\mathop{\limto}}

\newcommand{\leftlim}{\mathop{\displaystyle\lim_{\longleftarrow}}}
\newcommand{\limfromn}{\leftlim\limits_{\raise3pt\hbox{$n$}}}
\newcommand{\limton}{\rightlim\limits_{\raise3pt\hbox{$n$}}}

\newcommand{\rightlimit}[1]{\mathop{\lim\limits_{\longrightarrow}}\limits%
                    _{\raise3pt\hbox{$\scriptstyle #1$}}}

\newcommand{\leftlimit}[1]{\mathop{\lim\limits_{\longleftarrow}}\limits%
                    _{\raise3pt\hbox{$\scriptstyle #1$}}}

\newcommand{\epi}{\twoheadrightarrow}
\newcommand{\iso}{\buildrel{\sim}\over{\longrightarrow}}

\DeclareMathOperator{\Aut}{{Aut}}
\DeclareMathOperator{\BT}{{BT}}
\DeclareMathOperator{\Coker}{{Coker}}

\DeclareMathOperator{\Ker}{{Ker}} \DeclareMathOperator{\id}{{id}}
\DeclareMathOperator{\im}{{Im}} 

\DeclareMathOperator{\Mor}{{Mor}}
 \DeclareMathOperator{\op}{{op}}

\DeclareMathOperator{\Spec}{{Spec}}

\DeclareMathOperator{\Stab}{{Stab}}
\DeclareMathOperator{\pNilp}{{p-Nilp}}

\theoremstyle{definition}

\numberwithin{equation}{section}

\begin{document}
\title{On the quotient of a groupoid by an action of a 2-group}
\author{Vladimir Drinfeld}
\address{University of Chicago, Department of Mathematics, Chicago, IL 60637}
 \thanks{Partially supported by NSF grant   DMS-2001425.}
 \dedicatory{To G\'erard Laumon with deepest admiration}

%\email{drinfeld@math.uchicago.edu}

\begin{abstract}
If $\sX$ is a groupoid equipped with an action of a 2-group $\sG$ then one has a 2-groupoid $\sX/\sG$. We describe the fibers of the functor from $\sX/\sG$ to the 1-grou\-po\-id $\pi_0(\sX )/\pi_0(\sG)$.
We also give an explicit model for $\sX/\sG$ in a certain situation. 

The work gives an abstract model for a certain 2-stack which provides a conjectural description of the $p$-adic completion of the stack of $n$-truncated Barsotti-Tate groups.
\end{abstract}
%\keywords{k}

\maketitle

If a 2-group $\sG$ acts on a groupoid $\sX$ then one can form the quotient 2-groupoid $\sX/\sG$. This general construction is discussed in \S\ref{s:intro}.
In \S\ref{s:the situation} we discuss the following special situation: $\sX$ is the underlying groupoid of a 2-group, and $\sG$ acts on $\sX$ by two-sided translations.

This elementary paper is motivated by the following: the sheafified version of the situation of \S\ref{s:the situation} occurs in the definition of the 2-stack from \cite[\S D.8.3]{D}, which provides a conjectural description of the $p$-adic completion of the stack of $n$-truncated Barsotti-Tate groups and its ``Shimurian'' analogs (see Appendix~\ref{s:Relation to BT} for more details). The result of \S\ref{ss:Describing the 2-groupoid} of this paper could be used to test Conjecture~D.8.4 from \cite{D}.

I first constructed the 2-groupoid from \S\ref{ss:Describing the 2-groupoid} by trial and error (I was motivated by potential applications to the Lau group scheme and the Lau gerbe, see \S\ref{sss:relation to Lau}). Then the idea of treating this explicit construction as a special case of the 2-groupoid 
$\sX/\sG$ was suggested to me by D.~Arinkin and N.~Rozenblyum. I am very grateful to them.

\section{The quotient 2-groupoid}  \label{s:intro}
\subsection{The question}
\subsubsection{2-groupoids and 2-groups}   
Recall that a 2-groupoid is a 2-category in which all 1-morphisms and 2-morphisms are invertible.
A \emph{$2$-group} is a 2-groupoid with a single object; equivalently, a 2-groupoid is a monoidal category in which all objects and morphisms are invertible.

\subsubsection{Quotient groupoids}   \label{sss:Quotient groupoids} 
If a group $G$ acts on a set $X$ then the \emph{quotient groupoid} (or groupoid quotient) $X/G$ is defined as follows: the set of objects is $X$, a morphism $x\to x'$ is an element $g\in G$ such that $gx=x'$, and the composition of morphisms is given by multiplication in $G$.

\subsubsection{Quotient 2-groupoids}   \label{sss:Quotient 2-groupoids} 
More generally, if a 2-group $\sG$ acts on a groupoid $\sX$ then one defines the  \emph{quotient 2-groupoid} $\tilde\sX=\sX/\sG$ as follows: 

(i) $\Ob\tilde\sX:=\Ob\sX$;

(ii) for $x_1,x_2\in\sX$ let $\MMor_{\tilde\sX} (x_1,x_2)$ be the following groupoid: its objects are pairs 
$$(g,f), \mbox{ where }g\in\sG,\quad f\in\Isom (x_2,gx_1),$$
and a morphism $(g,f)\to (g',f')$ is a morphism $g\to g'$ such that the corresponding morphism $gx_1\to g'x_1$ equals $f'f^{-1}$; 

(iii) the composition functor $\MMor_{\tilde\sX} (x_1,x_2)\times \MMor_{\tilde\sX} (x_2,x_3)\to \MMor_{\tilde\sX} (x_1,x_3)$ comes from the product in $\sG$.

\subsubsection{The question}   \label{sss:the question}
In the situation of \S\ref{sss:Quotient 2-groupoids}, 
let $\sX':=\pi_0(\sX)/\pi_0(\sG)$, 
where $\pi_0$ stands for the set of isomorphism classes of objects and $\pi_0(\sX)/\pi_0(\sG)$ is the quotient 1-groupoid in the sense of \S\ref{sss:Quotient groupoids}.
We have a canonical functor
\begin{equation}  \label{e:the functor}
\tilde\sX=\sX/\sG\to \sX'.
\end{equation}

It is easy to see that the functor \eqref{e:the functor} induces a surjection at the level of objects and at the level of 1-morphisms. 
So the obstruction to it being an equivalence is formed by the 2-groups
\begin{equation}  \label{e:the 2-group}
\Ker (\AAut_{\tilde\sX}x\to\Aut_{\sX'}\bar x), \quad x\in\sX,
\end{equation}
where $\bar x\in\pi_0(\sX)$ is the image of $x$.
Let us note that $\Aut_{\sX'}\bar x$ is just the stabilizer of $\bar x$ in~$\pi_0 (\sG )$; similarly, $\AAut_{\tilde\sX}x$ is the ``categorical stabilizer'' of $x$ in $\sG$.

The problem is to describe the 2-group \eqref{e:the 2-group}. This
will be done in Proposition~\ref{p:the answer} in terms of a certain crossed module.

\subsection{Recollections on crossed modules} \label{ss: recollections and notation}
\subsubsection{Crossed modules}      \label{sss:crossed modules} 
Recall that a {\em crossed module} $G^{\bullet}$ is a pair of groups $G^0,G^{-1}$ 
together with an action of $G^0$ on $G^{-1}$  
and a homomorphism $d: G^{-1}\to G^0$ satisfying certain identities. The image of $\gamma\in G^{-1}$ under the action of $g\in G^0$ is denoted by ${}^g \gamma$, and the identities are as follows:
\begin{equation} \label{e:crossed module1}
d({}^g \gamma) = g d(\gamma) g^{-1}, \quad \gamma\in G^{-1}, g\in G,
\end{equation}
\begin{equation}  \label{e:crossed module2}
{}^{d(\gamma)} \gamma' = \gamma\gamma '\gamma ^{-1}\quad  \gamma ,\gamma '\in G^{-1}.
\end{equation}

\subsubsection{The 2-group corresponding to a crossed module}  \label{sss:The 2-group} 
A crossed module $G^{\bullet}$ gives rise to a strict 2-group, which we denote by 
\begin{equation}
\Cone (G^{-1}\overset{d}\longrightarrow G^0).
\end{equation} 
As a groupoid, this is the quotient of $G^0$ by the action of $G^{-1}$ given by
$(\gamma ,g)\mapsto d(\gamma )\cdot g$; thus the set of objects is $G^0$, and for $g,g'\in G^0$ one has
\begin{equation}   \label{e:Mor(g,g')}
\Mor (g,g')=\{\gamma\in G^{-1}\,|\, d(\gamma )g=g'\}.
\end{equation}
The tensor product map $\Mor (g_1,g'_1)\times \Mor (g_2,g'_2)\to \Mor (g_1g_2,g'_1g'_2)$ is given by
\begin{equation}  \label{e:tensor product}
 (\gamma_1,\gamma_2)\mapsto \gamma_1\cdot {}^{g_1}\gamma_2 .
\end{equation}
The assignment $G^{\bullet}\mapsto\Cone (G^{-1}\overset{d}\longrightarrow G^0)$ gives an equivalence between the category of crossed modules and that of strict 2-groups (e.g., see \cite[Lemma~2.2]{L}).

\subsubsection{The abelian case}  \label{sss:abelian case} 
We say that a crossed module $G^{\bullet}$ is \emph{abelian} if $G^0,G^{-1}$ are abelian and the action of $G^0$ on $G^{-1}$ is trivial. In this case
the 2-group from \S\ref{sss:The 2-group} is the one associated in \cite[Expos\'e XVIII, \S 1.4]{SGA4} to the complex 
\begin{equation}   \label{e:the complex}
0\to G^{-1}\overset{d}\longrightarrow G^0\to 0,
\end{equation}
so there is no conflict between understanding $\Cone (G^{-1}\overset{d}\longrightarrow G^0)$ as a 2-group and the usual understanding as the complex \eqref{e:the complex}.

\subsection{Answering the question from \S\ref{sss:the question}}
\subsubsection{A crossed module related to the $\sG$-action}
We keep the notation of \S\ref{sss:Quotient 2-groupoids}-\ref{sss:the question}. Let $\pi_1(\sG):=\Aut (1_{\sG})$, where $1_{\sG}$ is the unit object of the 2-group $\sG$; it is well known that $\pi_1(\sG)$ is abelian. The $\sG$-action induces for each $x\in\sX$ a homomorphism 
\begin{equation} \label{e:the central homomorphism}
\phi_x:\pi_1(\sG)\to\Aut_{\sX}x, 
\end{equation}
which is functorial in $x$, i.e., for any $\psi\in\Mor (x,y)$ one has $\psi\circ \phi_x=\phi_y\circ\psi$. Applying this for $y=x$, we see that \emph{$\im \phi_x$ is contained in the center of $\Aut_{\sX}x$.}
So we can regard \eqref{e:the central homomorphism} as a crossed module in which the action of $\Aut_{\sX}x$ on $\pi_1(\sG )$ is trivial. Let $\Cone (\phi_x)$ be the corresponding  2-group (see \S\ref{sss:The 2-group}).

\begin{prop}   \label{p:the answer}
There is a canonical equivalence of 2-groups
\begin{equation}   \label{e:the equivalence}
\Cone (\pi_1(\sG)\overset{\phi_x}\longrightarrow\Aut_{\sX})\iso\Ker (\AAut_{\tilde\sX}x\to\Aut_{\sX'}\bar x).
\end{equation}
\end{prop}

\begin{proof}
By \S\ref{sss:Quotient 2-groupoids}, objects of $\AAut_{\tilde\sX}x$ are pairs $(g,f)$, where $g\in\sG$ and $f:x\iso gx$. A~morphism $(g,f)\to (g',f')$ is a morphism $g\to g'$ inducing $f'\circ f^{-1}:gx\to g'x$. The product in $\AAut_{\tilde\sX}x$ is
\[
(g_1,f_1)\cdot (g_2,f_2)=(g_1g_2,\tilde f_2\circ f_1),
\]
where $\tilde f_2:g_1x\to g_1g_2x$ comes from $f_2:x\to g_2x$.

The equivalence \eqref{e:the equivalence} takes $f\in\Aut_{\sX}x$ to the pair $(1_\sG ,f)\in\AAut_{\tilde\sX}x$; at the level of morphisms, it comes from
$\id :\pi_1(\sG)\to \pi_1(\sG)$.
\end{proof}

\subsection{Some corollaries}
Let $\tilde\sX^{\le 1}$ be the 1-truncation of $\tilde\sX$, i.e., the 1-groupoid obtained by replacing the groupoids $\MMor_{\tilde\sX} (x_1,x_2)$, $x_i\in\tilde\sX$, by their $\pi_0$'s.
The functor $\tilde\sX\to\sX'$ factors as $\tilde\sX\to\tilde\sX^{\le 1}\to\sX'$. The functor 
\begin{equation} \label{e:the gerbe}
\tilde\sX^{\le 1}\to\sX'
\end{equation}
 is a gerbe.\footnote{A functor between groupoids is said to be a gerbe if each of its fibers has one and only one isomorphism class of objects.}

\begin{cor}   \label{c:Ker&Coker}
For each $x\in\sX$, one has

(i) $\Ker (\Aut_{\tilde\sX^{\le 1}}(x)\to\Aut_{\sX'}(x))=\Coker \phi_x$.

(ii) $\Aut_{\sX}(\id_x)=\Ker f_x$.
\end{cor}

\subsubsection{The abelian situation}   \label{sss:abelian situation} 
Suppose that for each $x\in\sX$ the group $\pi_1 (\sX ,x)=\Aut_{\sX}x$ is abelian. Then $\pi_1 (\sX ,x)$ and the homomorphism \eqref{e:the central homomorphism} depend only on $\bar x\in\pi_0(\sX )$, so we can rewrite \eqref{e:the central homomorphism} as
\begin{equation}  \label{e:2the central homomorphism}
\phi_{\bar x}:\pi_1(\sG)\to\pi_1 (\sX ,\bar x), \quad \bar x\in\pi_0(\sX ).
\end{equation}
Let $\fL ({\bar x}):=\Coker\phi_{\bar x}$.

The group $\pi_0(\sG)$ acts on $\pi_1(\sG)$ and $\pi_0(\sX)$, and the collection of homomorphisms \eqref{e:2the central homomorphism} is $\pi_0(\sG)$-equivariant. So the collection of abelian groups $\fL ({\bar x})$, $\bar x\in\pi_0(\sX)$, is $\pi_0(\sG)$-equivariant. In other words, we get a functor
\begin{equation}   \label{e:Lau functor}
\fL: \sX'\to\Ab;
\end{equation}
as before, $\sX'$ denotes the quotient groupoid $\pi_0(\sX)/\pi_0 (\sG )$.

The gerbe \eqref{e:the gerbe} is banded by the functor $\fL$; this follows from Corollary~\ref{c:Ker&Coker}(i).

\section{A particular situation}   \label{s:the situation}
\subsection{Subject of this section}  \label{ss:the situation}
\subsubsection{}   \label{sss:the situation}
If $\fG$ is a 2-group and $\sX$ is the underlying groupoid of $\fG$ then $\fG\times\fG$ acts on $\sX$ by two-sided translations:  namely, $(g,g')\in\fG\times\fG$ acts by $x\mapsto gx(g')^{-1}$.
So given a homomorphism of 2-groups $\fB\to\fG\times\fG$,
we get an action of $\fB$ on $\sX$ and the corresponding 2-groupoid~$\sX/\fB$.

\subsubsection{The goal}  \label{sss:the cone}
Now let $B^{\bullet}, G^{\bullet}$ be crossed modules and $\pi,\pi':B^{\bullet}\to G^{\bullet}$ be homomorphisms. Then we get the homomorphism $(\pi,\pi'):\fB\to\fG\times\fG$, where $\fB:=\Cone (B^{\bullet})$, $\fG:=\Cone (G^{\bullet})$. 
It gives a strict action of the strict 2-group $\fB$ on $\sX$ and therefore a strict 2-groupoid~$\sX/\fB$. We will denote this strict 2-groupoid by 
\begin{equation}  \label{e:the cone}
\Cone (B^{\bullet}\overset{\pi,\pi'}\longrightarrow G^{\bullet}).
\end{equation}
The main goal of \S\ref{s:the situation} is to give a certain tautological reformulation of the construction of the 2-groupoid \eqref{e:the cone}. This will be done in \S\ref{ss:Describing the 2-groupoid}. Let us note two cases in which this reformulation looks nice.

\subsubsection{Two easy cases}  \label{sss:Two easy cases}
(i) Suppose that $B^{\bullet}$ and $G^{\bullet}$ are abelian in the sense of \S\ref{sss:abelian case}. Then
$\Cone (B^{\bullet}\overset{\pi-\pi'}\longrightarrow G^{\bullet})$
is a complex of abelian groups $0\to C^{-2}\overset{d}\longrightarrow C^{-1}\overset{d}\longrightarrow C^0\to 0$, 
and $\Cone (B^{\bullet}\overset{\pi,\pi'}\longrightarrow G^{\bullet})$ is the strict 2-group associated to this complex
in the usual way: its objects are elements of $C^0$, and for $c,c'\in C^0$ the groupoid of morphisms $c\to c'$ is the quotient of the set
$\{x\in C^{-1}\,|\, dx=c'-c\}$ by the action of $C^{-2}$.

(ii) If $B^{-1}=0$ then $\Cone (B^{\bullet}\overset{\pi,\pi'}\longrightarrow G^{\bullet})$ is a 1-groupoid.  To describe it, first note that the maps
\[
B^0\times G^0\to G^0, \quad (b,g)\mapsto \pi (b)g\pi'(b)^{-1},
\]
\[
G^{-1}\times G^0\to G^0, \quad (\gamma,g)\mapsto d(\gamma )g
\]
define actions of the groups $B^0$ and $G^{-1}$ on the set $G^0$. These actions combine into an action of $B^0\ltimes_{\pi}G^{-1}$ on the set $G^0$, where $B^0\ltimes_{\pi}G^{-1}$ is the semidirect product via the homomorphism
$B^0\overset{\pi}\longrightarrow G^0\to\Aut G^{-1}$. It is easy to check that the corresponding quotient groupoid (in the sense of \S\ref{sss:Quotient groupoids}) is $\Cone (B^{\bullet}\overset{\pi,\pi'}\longrightarrow G^{\bullet})$.

\subsubsection{}
In general, it is easy to check that the objects and 1-morphisms of $\Cone (B^{\bullet}\overset{\pi,\pi'}\longrightarrow G^{\bullet})$ are the same as in \S\ref{sss:Two easy cases}(ii).
To describe the 2-morphisms, we use a slight generalization of the notion of crossed module, see \S\ref{ss:X-crossed modules} below.

\subsection{The maps $\phi_g$}   \label{ss:the maps phi_g}
Let us describe the map \eqref{e:the central homomorphism}. In our situation, $\sG =\Cone (B^\bullet )$ and $\sX$ is the underlying groupoid of $\Cone (G^\bullet )$. Note that 
$$\pi_1 (\sG )=H^{-1}(B^\bullet ):=\Ker (B^{-1}\overset{d}\longrightarrow B^0).$$
We have $\Ob\sX=G^0$, and for $g\in G^0$ the group $\Aut_\sX g$ identifies with $H^{-1}(G^\bullet )$ via \eqref{e:Mor(g,g')} (in particular, $\Aut_\sX g$ is abelian, so we are in the situation of \S\ref{sss:abelian situation}).
Thus the homomorphism \eqref{e:the central homomorphism} is a map
\[
\phi_g: H^{-1}(B^\bullet )\to H^{-1}(G^\bullet ).
\]

\begin{lem}  \label{l:the map phi_g}
The homomorphism $\phi_g$ is as follows:
\begin{equation}   \label{e:phi_g}
\phi_g (\beta)=\pi (\beta)\cdot {}^{g}\pi'(\beta)^{-1}={}^{g}\pi'(\beta)^{-1}\cdot \pi (\beta), \quad \beta\in H^{-1}(B^\bullet ).
\end{equation}
\end{lem}

\begin{proof}
By \eqref{e:tensor product}, the map
\[
\Mor (g_1,g'_1)\times \Mor (g_2,g'_2)\times \Mor (g_3,g'_3)\to \Mor (g_1g_2g_3,g'_1g'_2g'_3)
\]
is given by $(\gamma_1,\gamma_2,\gamma_3)\mapsto \gamma_1\cdot {}^{g_1}\gamma_2\cdot {}^{g_1g_2}\gamma_3$. To get $\phi_g (\beta)$, one has to take
$g_1=g_3=1$, $g_2=g$, $\gamma_1=\pi(\beta )$, $\gamma_2=1$, $\gamma_3=\pi'(\beta)^{-1}$.
\end{proof}

\subsubsection{Remarks}  \label{sss:relation to Lau}
(i) Lemma~\ref{l:the map phi_g} implies that in our situation the functor $\fL$ from \S\ref{sss:abelian situation} is $g\mapsto\Coker\phi_g$, where
$\phi_g$ is given by \eqref{e:phi_g}.

(ii) The author hopes that the functor $\fL$ from the previous remark can serve as an abstract model of the Lau group scheme (in the sense of \cite[Thm.~1.1.1]{D}) and that 
the canonical gerbe banded by~$\fL$ (see \eqref{e:the gerbe} and \S\ref{sss:abelian situation}) can serve as an abstract model of the Lau gerbe (by which we mean the gerbe from \cite[Thm.~1.1.1]{D}).

\subsection{A way to describe strict 2-groupoids}   \label{ss:X-crossed modules}
A strict 2-groupoid with a single object (a.k.a. a strict 2-group) is ``the same as'' a crossed module, see \S\ref{sss:crossed modules}-\ref{sss:The 2-group}.
Arbitrary strict 2-groupoids have a similar description via a slight generalization of the notion of crossed module.

\subsubsection{Definition}  \label{sss:X-crossed module}
Let $X$ be a set. An \emph{$X$-crossed module} is the following data:

(i) a groupoid $\Gamma$ with $\Ob\Gamma=X$; 

(ii) a functor
\[
\Gamma\mapsto\{\mbox{Groups}\}, \quad x\mapsto H_x ;
\]

(iii) a collection of homomorphisms
\[
d_x: H_x\to\Aut_\Gamma x
\]
such that $d_x$ is functorial in $x$ and
\[
{}^{d_x(h)} h' = hh 'h ^{-1}\quad   \mbox{ for all } x\in X  \mbox{ and }h ,h '\in H_x,
\]
where ${}^{d_x(h)} h'$ stands for the image of $h'$ under the automorphism of $H_x$ corresponding to $d_x(h)\in \Aut_\Gamma x$ by functoriality of $H_x$.

If $X$ has a single element one gets the usual notion of crossed module, see \S\ref{sss:crossed modules}.

\subsubsection{From strict 2-groupoids to $X$-crossed modules}
Let $\cC$ be a strict 2-groupoid and $X=\Ob\cC$. Then one defines an $X$-crossed module as follows:

(i) $\Gamma$ is the 1-skeleton of $\cC$ (i.e., the 1-groupoid obtained by disregarding the 2-morphisms of $\cC$);

(ii) for $x\in X$ one sets $H_x:=\Ker (\Aut_\Gamma x\epi \AAut_{\cC} x)$, where $\Ker$ stands for the categorical fiber over $\id_x\in\AAut_{\cC} x$; so $H_x$ is formed by pairs $(g,f)$, where $g\in\Aut_{\Gamma}x$ and $f:g\iso\id_{\cC}x$ is a 2-morphism in $\cC$;

(iii) the map $d_x: H_x\to\Aut_\Gamma x$ forgets $f$.

\medskip

We have defined a functor from the 1-category of strict 2-groupoids to the category of pairs consisting of a set $X$ and an $X$-crossed module. This functor is an equivalence, and the inverse functor is described below.

\subsubsection{From $X$-crossed modules to strict 2-groupoids}   \label{sss:From X-crossed modules to 2-groupoids}
In the situation of \S\ref{sss:X-crossed module} we have to define the groupoids
$\MMor_{\cC} (x,x')$ for $x,x'\in X$ and the composition functors 
\begin{equation}  \label{e:2tensor product}
 \MMor (x',x'') \times\MMor (x,x')\to\MMor (x,x'').
\end{equation}
for $x,x',x''\in X$. 

$\MMor_{\cC} (x,x')$ is defined to be the groupoid quotient  of the set $\Mor_\Gamma (x,x')$ with respect to the action of $H_{x'}$ that comes from the homomorphism $d_{x'}: H_x\to\Aut_\Gamma x'$.
Thus a morphism in $\MMor (x,x')$ is a triple $$\alpha =(g,\tilde g,h),$$ where $g,\tilde g\in\Mor (x,x')$ are the source and target of $\alpha$, and $h\in H_{x'}$ is such that $\tilde g=d_{x'}(h)\cdot g$.

At the level of objects, the functor \eqref{e:2tensor product} is just the composition map in the groupoid $\Gamma$. Let us define \eqref{e:2tensor product} at the level of morphisms.
Let $\alpha_1$ (resp.~$\alpha_2$) be a morphism in $\MMor (x',x'')$ (resp.~$\MMor (x,x')$). As above, write $\alpha_i$ as a triple $(g_i,\tilde g_i,h_i)$. The image of $(\alpha_1,\alpha_2)$ under \eqref{e:2tensor product} is defined to be the triple
\[
(g_1g_2,\,\tilde g_1\tilde g_2,\, h_1\cdot {}^{g_1} h_2)
\]
similarly to formula \eqref{e:tensor product}.

\subsection{Describing the 2-groupoid \eqref{e:the cone}} \label{ss:Describing the 2-groupoid}
Let $X$ be the underlying set of $G^0$. Let us describe the $X$-crossed module corresponding to the strict 2-groupoid 
$\Cone (B^{\bullet}\overset{\pi,\pi'}\longrightarrow G^{\bullet})$ (the 2-groupoid itself can be recovered from the $X$-crossed module as explained in \S\ref{sss:From X-crossed modules to 2-groupoids}).

The 1-groupoid $\Gamma$ is the one corresponding to the action of the group $B^0\ltimes_{\pi}G^{-1}$ on $X$ described in \S\ref{sss:Two easy cases}(ii). It remains for us to describe the data from
\S\ref{sss:X-crossed module}(ii-iii), i.e., the ($B^0\ltimes_{\pi}G^{-1}$)-equi\-variant family of groups $H_g$, $g\in G^0$, and the homomorphisms 
$$d_g:H_g\to\Stab_g, \quad g\in G^0,$$
where $\Stab_g\subset B^0\ltimes_{\pi}G^{-1}$ is the stabilizer of $g$, i.e.,
\begin{equation}   \label{e:Stab_g}
\Stab_g=\{ b\cdot\gamma\,|\,b\in B^0,\, \gamma\in G^{-1}, \, d(\gamma )=\pi (b)^{-1}\cdot g\pi'(b)g^{-1}\}.
\end{equation}

 It is easy to check that these data are as follows:

(i) $H_g=B^{-1}$, and the $B^0\ltimes_{\pi}G^{-1}$-equivariant structure comes from the action of $B^0$ on~$B^{-1}$;

(ii) in terms of \eqref{e:Stab_g}, the map $d_g:B^{-1}\to \Stab_g$ is given by $b=d(\beta )$, $\gamma =\pi (\beta)^{-1}\cdot {}^{g}\pi'(\beta)$, where $\beta\in B^{-1}$; so
\begin{equation}   \label{e:delta_g}
d_g (\beta)=d(\beta)\cdot \pi (\beta)^{-1}\cdot {}^{g}\pi'(\beta), \quad \beta\in B^{-1}.
\end{equation}

\subsubsection{Remark}
It is easy to check that the map 
$$d_g:B^{-1}\to B^0\ltimes_{\pi}G^{-1}, \quad \beta\mapsto d(\beta)\cdot \pi (\beta)^{-1}$$
is a homomorphism whose image centralizes $G^{-1}$. So the map \eqref{e:delta_g} is a homomorphism. Moreover,
one could rewrite \eqref{e:delta_g} as $d_g (\beta)={}^{g}\pi'(\beta)\cdot  d(\beta)\cdot\pi (\beta)^{-1}$.

%\newpage

\appendix

\section{Relation to the 2-stack $\BT_n^{G,\mu,?}$}  \label{s:Relation to BT}
\subsection{Goal of this Appendix}
Let $\BB^{\bullet},\BG^{\bullet}$ be crossed modules in some topos. Suppose we are given homomorphisms $\pi,\pi':\BB^{\bullet}\to\BG^{\bullet}$.  Then one defines a 2-stack
\begin{equation}  \label{e:2the cone}
\Cone (\BB^{\bullet}\overset{\pi,\pi'}\longrightarrow \BG^{\bullet})
\end{equation}
similarly to \S\ref{sss:the cone}; if the topos is a point then \eqref{e:2the cone} is the 2-groupoid~\eqref{e:the cone}.

Under a very mild assumption\footnote{See \S\ref{sss:To be explained} below.}, the 2-stack $\BT_n^{G,\mu,?}$ from \cite[\S D.8.3]{D} can be writtten in the form \eqref{e:2the cone} in a rather natural way. The goal of this Appendix is to provide some details about this in a somewhat informal way.

\subsection{The topos}
Throughout this Appendix, we fix a prime $p$.
A ring $R$ is said to be $p$-nilpotent if the element $p\in R$ is nilpotent. Let $\pNilp$ denote the category of $p$-nilpotent rings.
The topos relevant for us is the category of fpqc-sheaves of sets on $\pNilp^{\op}$. From now on, the word ``stack'' will refer to this topos.

\subsection{The stack $\BT_n^{G,\mu}$}
\subsubsection{The results of \cite{GM}}
Let $n\in\BN$.
Let $G$ be a smooth affine group scheme over $\BZ/p^n\BZ$ and $$\mu :\BG_m\to G$$ a cocharacter.

%Let $\BT_n^{G,\mu}$ be the stack defined in \cite[\S 9]{GM}, so if $R\in\pNilp$ then $\BT_n^{G,\mu}(R)$ is the groupoid of $G$-bundles on $R^{\Syn}\otimes (\BZ/p^n\BZ)$ satisfying a certain condition, which depends on $\mu$.

For $R\in \pNilp$ let $\BT_n^{G,\mu}(R)$ be as in \cite[\S 9]{GM}; this is the $\infty$-groupoid\footnote{If $R$ is good enough (e.g., l.c.i) then the derived stack $R^{\Syn}\otimes (\BZ/p^n\BZ)$ is classical, so $\BT_n^{G,\mu}(R)$ is a 1-groupoid.} 
%For $R\in \pNilp$ let $\BT_n^{G,\mu}(R)$ be as in \cite[\S 9]{GM}; this is the $\infty$-groupoid\footnote{If $R$ is good enough (e.g., l.c.i) then the derived stack } of $G$-bundles on } 
of $G$-bundles on $R^{\Syn}\otimes (\BZ/p^n\BZ)$ satisfying a certain condition, which depends on $\mu$. Here $R^{\Syn}$ is the syntomification of $R$.

$\mu$ is said to be \emph{$1$-bounded} if all weights of the action of $\BG_m$ on $\Lie (G)$ are $\le 1$ (if $G$ is reductive and almost simple this means that $\mu$ is minuscule or zero).
By \cite[Thm.~D]{GM}, if $\mu$ is $1$-bounded then $\BT_n^{G,\mu}$ is a smooth algebraic stack over $\Spf\BZ_p$; in other words, for every $m\in\BN$ the restriction of $\BT_n^{G,\mu}$ to the category of $\BZ/p^m\BZ$-algebras is a smooth algebraic 1-stack over $\BZ/p^m\BZ$.
By \cite[Thm.~A]{GM}, if $G=GL(d)$ and $\mu$ is $1$-bounded then $\BT_n^{G,\mu}$ identifies with the stack of $n$-truncated Barsotti-Tate groups of height $d$ and dimension $d'$, where $d'$ depends on $\mu$.

\subsubsection{The 2-stack $\BT_n^{G,\mu,?}$ and the conjecture}
\S D.8.3 of \cite{D} contains a definition of a certain 2-stack $\BT_n^{G,\mu,?}$; in this definition $\mu$ is not required to be $1$-bounded. Conjecture~D.8.4 from \cite{D} says that if $\mu$ is $1$-bounded then $\BT_n^{G,\mu}=\BT_n^{G,\mu,?}$ (which implies that $\BT_n^{G,\mu,?}$ is a 1-stack if $\mu$ is $1$-bounded).

\subsubsection{To be explained below}    \label{sss:To be explained}
Assume that $G$ is lifted to a smooth affine group scheme over~$\BZ_p$ (note that such a lift is automatic if $G$ is reductive). Our goal is to explain why
$\BT_n^{G,\mu,?}$ can be rather naturally written in the form \eqref{e:2the cone}, where $\BB^i$ and $\BG^i$ are explicit group ind-schemes over $\BZ_p$.
%$\Spf\BZ_p$. Moreover, in the case $i=0$ they are \emph{schemes} over $\Spf\BZ_p$ (by this we mean that the morphisms $\BB^0\to\Spf\BZ_p$ and $\BG^0\to\Spf\BZ_p$ are schematic).

%\newpage

\subsection{Format of the definition of $\BT_n^{G,\mu,?}$}
%The definition of the 2-stack $\BT_n^{G,\mu,?}$ given in \cite[\S D.8.3]{D} has the format from \S\ref{sss:the situation}
\subsubsection{The story in a few words}
According to the definition from \cite[\S D.8.3]{D}, the 2-stack $\BT_n^{G,\mu,?}$ is obtained by applying the construction of \S\ref{sss:the situation} to a certain homomorphism of group stacks 
\begin{equation}   \label{e:morphism of group stacks}
\fB\to\fG\times\fG,
\end{equation}
see formula~(D.10) of \cite{D}.
It turns out that this homomorphism lifts in a rather natural way to a homomorphism 
\[
\BB^{\bullet}\overset{(\pi,\pi')}\longrightarrow\BG^{\bullet}\times\BG^{\bullet}
\]
of fpqc sheaves of crossed modules. Such a lift provides a realization of $\BT_n^{G,\mu,?}$ in the form~\eqref{e:2the cone}.

\subsubsection{What will be explained}
Instead of discussing diagram \eqref{e:morphism of group stacks}, we will only discuss $\fG$: we will recall the definition of the 2-stack $\fG$ from \cite{D} and explain how to lift it to a sheaf of crossed modules $\BG^{\bullet}$.
Let us note that $\fG$ does not depend on the cocharacter $\mu$ (but $\fB$ does).

\subsubsection{The 2-stack $\fG$}
By definition, 
\begin{equation}  \label{e:who is fG}
\fG :=G(\SR_n),
\end{equation}
where $\SR_n$ is a certain stack of $\BZ/p^n\BZ$-algebras, whose definition is sketched\footnote{The stack $\SR_{n,\BF_p}:=\SR_n\times\Spec\BF_p$ is \emph{completely described} in \cite[\S D.7.2]{D}.} in \cite[\S D.7.1]{D}.
Formula~\eqref{e:who is fG} just means that 
\begin{equation}  \label{e:2who is fG}
\fG (A):=G (\SR_n (A)),   \quad A\in\pNilp .
\end{equation}
The r.h.s of \eqref{e:2who is fG} makes sense because $G$ is a scheme over $\BZ/p^n\BZ$ and $\SR_n (A)$ is an animated $\BZ/p^n\BZ$-algebra.

\subsection{Constructing the crossed module $\BG^{\bullet}$}
\subsubsection{Recollections}   \label{sss:Recollections}
In addition to \S\ref{sss:crossed modules} and the interpretation via strict 2-groups at the end of \S\ref{sss:The 2-group}, there are two other well known points of view on crossed modules:

(i) a crossed module is the same as a 2-group $\fG$ with an epimorphism $G^0\epi\fG$, where $G^0$ is a group; 
%\Drin{should I say that $G^1=\Ker (G^0\epi\fG)$?}

(ii) a crossed module is the same as a groupoid\footnote{The nerve of this groupoid is the \v {C}ech nerve of the epimorphism $G^0\epi\fG$.} internal to the category of groups.

\subsubsection{Strategy for constructing $\BG^{\bullet}$}
Suppose that $G$ is lifted to a smooth affine group scheme $\tilde G$ over~$\BZ_p$. Suppose that we have a ring scheme $\cA$ over $\Spf\BZ_p$ equipped with an epimorphism
\begin{equation}  \label{e:cA to hat sR_n}
\cA\epi\SR_n.
\end{equation}
Set $\BG^0:=\tilde G (\cA )$.
Then we get a homomorphism 
\begin{equation}  \label{e:BG^0}
\BG^0=\tilde G (\cA )\to \tilde G(\SR_n )=G(\SR_n )=\fG .
\end{equation}
Using smoothness of $\tilde G$, one checks that it is surjective (in the sense of fpqc sheaves).
So one gets a crossed module $\BG^{\bullet}$ by applying \S\ref{sss:Recollections}(i) to the homomorphism \eqref{e:BG^0}.
One has $\BG^{-1}=\Ker (\BG^0\epi\fG )$.

\subsubsection{A nice crossed module}
There exists an epimorphism \eqref{e:cA to hat sR_n} with $\cA=W_n$, where $W_n$ is the ring scheme\footnote{$W_n$ is \emph{not} a scheme of $\BZ/p^n\BZ$-algebras (this is why we need $\tilde G$). On the other hand, $W_{n,\BF_p}$ \emph{is} a scheme of $\BZ/p^n\BZ$-algebras.} of $n$-truncated $p$-typical Witt vectors.
Moreover, if $p>2$ then there is a very nice epimorphism 
\begin{equation}   \label{e:nice epimorphism}
W_n\epi\SR_n
\end{equation}
 whose kernel equals $\hat W^{(F^n)}:=\Ker (F^n:\hat W\to\hat W)$. Here $\hat W\subset W$ is the following ind-scheme: for any $A\in\pNilp$, the ideal
$\hat W(A)$ is the set of all $x\in W(A)$ such that all components of the Witt vector $x$ are nilpotent and all but finitely many of them are zero.

The crossed module $\BG^{\bullet}$ corresponding to \eqref{e:nice epimorphism} is very simple: $\BG^0=\tilde G (W_n)$, $\BG^{-1}=\tilde G (\hat W^{(F^n)})$, $d:\BG^{-1}\to\BG^0$ comes from the canonical map
$\hat W^{(F^n)}\to W_n$, and the action of $\BG^0$ on $\BG^{-1}$ comes from the equality $\BG^{-1}=\Ker (\tilde G(W_n\ltimes\hat W^{(F^n)} )\epi \tilde G(W_n))$, where $W_n\ltimes\hat W^{(F^n)}$ is the semidirect product.

%\subsection{Quasi-ideal pairs}

\bibliographystyle{alpha}

\end{document}